\documentclass[12pt]{amsart}
\usepackage{amsfonts,amssymb,amscd,amsmath,enumerate,verbatim}
\usepackage[latin1]{inputenc}
\usepackage{amscd}
\usepackage{latexsym}
\usepackage{mathptmx}
\usepackage{multicol}
\input xy
\xyoption{all}

\def\NZQ{\mathbb}               

\def\ZZ{{\NZQ Z}}
\def\RR{{\NZQ R}}

\def\frk{\mathfrak}               

\def\Phi{{\frk N}}
\def\ab{{\bold a}}

\def\eb{{\bold e}}
\def\tb{{\bold t}}
\def\xb{{\bold x}}

\def\opn#1#2{\def#1{\operatorname{#2}}} 
\opn\chara{char} 
\opn\length{\ell} 
\opn\pd{pd} 
\opn\rk{rk}
\opn\projdim{proj\,dim} 
\opn\injdim{inj\,dim} 
\opn\rank{rank}
\opn\depth{depth} 
\opn\grade{grade} 
\opn\height{height}
\opn\embdim{emb\,dim} 
\opn\codim{codim}

\opn\Tr{Tr} 
\opn\bigrank{big\,rank}
\opn\superheight{superheight}
\opn\lcm{lcm}
\opn\trdeg{tr\,deg}
\opn\reg{reg} 
\opn\lreg{lreg} 
\opn\ini{in} 
\opn\lpd{lpd}
\opn\size{size}
\opn\mult{mult}
\opn\dist{dist}
\opn\cone{cone}
\opn\lex{lex}
\opn\rev{rev}
\opn\codeg{codeg}
\opn\div{div} \opn\Div{Div} \opn\cl{cl} \opn\Cl{Cl}
\opn\Spec{Spec} \opn\Supp{Supp} \opn\supp{supp} \opn\Sing{Sing}
\opn\Ass{Ass} \opn\Min{Min}
\opn\Ann{Ann} \opn\Rad{Rad} \opn\Soc{Soc}
\opn\Syz{Syz} \opn\Im{Im} \opn\Ker{Ker} \opn\Coker{Coker}
\opn\Am{Am} \opn\Hom{Hom} \opn\Tor{Tor} \opn\Ext{Ext}
\opn\End{End} \opn\Aut{Aut} \opn\id{id} \opn\ini{in}

\opn\nat{nat}
\opn\pff{pf}
\opn\Pf{Pf} \opn\GL{GL} \opn\SL{SL} \opn\mod{mod} \opn\ord{ord}
\opn\Gin{Gin}
\opn\Hilb{Hilb}\opn\adeg{adeg}\opn\std{std}\opn\ip{infpt}
\opn\Pol{Pol}
\opn\sat{sat}
\opn\Var{Var}
\opn\Gen{Gen}

\opn\aff{aff} \opn\con{conv} \opn\relint{relint} \opn\st{st}
\opn\lk{lk} \opn\cn{cn} \opn\core{core} \opn\vol{vol}
\opn\link{link} \opn\star{star}
\opn\gr{gr}

\def\Hc{{\mathcal H}}

\def\Pc{{\mathcal P}}
\def\Qc{{\mathcal Q}}

\def\pot#1#2{#1[\kern-0.28ex[#2]\kern-0.28ex]}

\opn\dirlim{\underrightarrow{\lim}}
\opn\inivlim{\underleftarrow{\lim}}
\let\iso=\cong

\let\to=\rightarrow

\def\Implies{\ifmmode\Longrightarrow \else
        \unskip${}\Longrightarrow{}$\ignorespaces\fi}
\def\implies{\ifmmode\Rightarrow \else
        \unskip${}\Rightarrow{}$\ignorespaces\fi}
\def\iff{\ifmmode\Longleftrightarrow \else
        \unskip${}\Longleftrightarrow{}$\ignorespaces\fi}

\let\:=\colon
\newtheorem{Theorem}{Theorem}[section]
\newtheorem{Lemma}[Theorem]{Lemma}
\newtheorem{Corollary}[Theorem]{Corollary}

\newtheorem{Example}[Theorem]{Example}

\newtheorem{Conjecture}[Theorem]{Conjecture}

\let\epsilon\varepsilon
\let\phi=\varphi
\let\kappa=\varkappa
\def\qed{\ifhmode\textqed\fi
      \ifmmode\ifinner\quad\qedsymbol\else\dispqed\fi\fi}
\def\textqed{\unskip\nobreak\penalty50
       \hskip2em\hbox{}\nobreak\hfil\qedsymbol
       \parfillskip=0pt \finalhyphendemerits=0}
\def\dispqed{\rlap{\qquad\qedsymbol}}

\opn\dis{dis}
\opn\height{height}
\opn\dist{dist}
\def\pnt{{\raise0.5mm\hbox{\large\bf.}}}

\opn\Lex{Lex}

\begin{document}
\title{Edge rings with $3$-linear resolutions}
\author{Takayuki Hibi, Kazunori Matsuda and Akiyoshi Tsuchiya}
\address{Takayuki Hibi,
Department of Pure and Applied Mathematics,
Graduate School of Information Science and Technology,
Osaka University, Suita, Osaka 565-0871, Japan}
\email{hibi@math.sci.osaka-u.ac.jp}
\address{Kazunori Matsuda,
Kitami Institute of Technology,
Kitami, Hokkaido 090-8507, Japan}
\email{kaz-matsuda@mail.kitami-it.ac.jp}
\address{Akiyoshi Tsuchiya,
Department of Pure and Applied Mathematics,
Graduate School of Information Science and Technology,
Osaka University, Suita, Osaka 565-0871, Japan}
\email{a-tsuchiya@ist.osaka-u.ac.jp}
\thanks{The authors are partially supported by JSPS KAKENHI 26220701, 17K14165 and 16J01549.}
\begin{abstract}
It is shown that the edge ring of a finite connected simple graph with 
a $3$-linear resolution is a hypersurface.
\end{abstract}
\maketitle
\section*{Introduction}
The edge ring and the edge polytope of a finite connected simple graph together with its toric ideal has been studied by many articles.  Their foundation was established in \cite{OH, 2linear}.  In \cite[Theorem 4.6]{2linear} it is shown that the edge ring $K[G]$, where $K$ is a field, of a finite connected simple graph $G$ on $[N] = \{ 1, \ldots, N \}$ has a $2$-linear resolution if and only if $K[G]$ is isomorphic to the polynomial ring in $N - \delta$ variables over the Segre product $K[x_1, x_2] \sharp K[y_1, \ldots, y_\delta]$ of two polynomial rings $K[x_1, x_2]$ and $K[y_1, \ldots, y_\delta]$, where $\delta$ is the normalized volume (\cite[p.~36]{Stu}) of the edge polytope $\Pc_G$ of $G$.  The purpose of the present paper is to study the question when $K[G]$ has a $3$-linear resolution.   

\begin{Theorem}
\label{thm:lin}
Let $G$ be a finite connected simple graph and $K[G]$ its edge ring.  If $K[G]$ has a $3$-linear resolution, then $K[G]$ is a hypersurface. 
\end{Theorem}

Achieving our proof of Theorem \ref{thm:lin}, we cannot overcome the temptation to give the following 

\begin{Conjecture}
	\label{thm:CON}
The edge ring of a finite connected simple graph with a $q$-linear resolution, where $q \geq 3$, is a hypersurface.
\end{Conjecture}
 
When the edge ring $K[G]$ of a finite simple graph $G$ is studied, we follow the convention of assuming that $G$ is connected.  Let $G$ be a finite disconnected simple graph with the connected components $G_1, G_2, \ldots, G_s$ and suppose that each $G_i$ has at least one edge.  Then the edge ring of $G$ is 
$
K[G] = K[G_{1}] \otimes_{K} \cdots \otimes_{K} K[{G_{s}}]
$ and its toric ideal is
\[
(I_{G_1}, I_{G_2}, \ldots, I_{G_s})
\subset K[x_1^{(1)}, \ldots, x_{n_1}^{(1)}, x_1^{(2)}, \ldots, x_{n_2}^{(2)}, \ldots, x_1^{(s)}, \ldots, x_{n_s}^{(s)}].
\]
Let, say, $I_{G_1} \neq (0)$ and $I_{G_2} \neq (0)$.  Then $K[G]$ cannot have a linear resolution.  Hence $K[G]$ has a $d$-linear resolution if and only if there is $1 \leq i \leq s$ for which $K[G_i]$ has a $d$-linear resolution and each $K[G_j]$ with $i \neq j$ is the polynomial ring.  

In the present paper, after preparing necessary materials on edge polytopes and edge rings (Section $1$), regularity and linear resolutions (Section $2$), and $\delta$-polynomials and degrees of lattice polytopes (Section $3$), Theorem \ref{thm:lin} will be proved in Section $4$.

\section{Edge polytopes and Edge rings}
A {\em lattice polytope} is a convex polytope all of whose coordinates have integer coordinates.  Let $\Pc \subset \RR^N$ be a lattice polytope of dimension $d$ and $\Pc \cap \ZZ^N=\{\ab_1,\ldots,\ab_n \}$.  Let $K$ be a field and $K[\tb^{\pm1}, s]=K[t_1^{\pm1},\ldots,t_N^{\pm 1},s]$ the Laurent polynomial ring in $N + 1$ variables over $K$.  Given a lattice point  $\ab=(a_1,\ldots,a_N) \in \ZZ^N$, we write $\tb^{\ab}$ for the Laurent monomial $t_1^{a_1}\cdots t_N^{a_N} \in K[\tb^{\pm 1}, s]$.  The \textit{toric ring} $K[\Pc]$ of $\Pc$ is the subalgebra of $K[\tb^{\pm 1},s]$ which is generated by those monomials $\tb^{\ab_1}s,\ldots,\tb^{\ab_n}s$ over $K$.  Let $K[\xb]=K[x_1,\ldots,x_n]$ denote the polynomial ring in $n$ variables over $K$ and define the surjective ring homomorphism $\pi:K[\xb] \to K[\Pc]$ by setting $\pi(x_i)=\tb^{\ab_i}s$ for $1 \leq i \leq n$.  The kernel $I_{\Pc}$ of $\pi$ is called the \textit{toric ideal} of $\Pc$.

Let $G$ be a finite connected simple graph on the vertex set $V(G)=[N]$ with the edge set $E(G)=\{e_1,\ldots,e_n\}$.  Let $\eb_1,\ldots,\eb_N$ denote the canonical unit coordinate vectors of $\RR^N$.  Given an edge $e=\{i,j\}$ of $G$, we set $\rho(e)=\eb_i+\eb_j \in \RR^{N}$.  The \textit{edge polytope} $\Pc_G$ of $G$ is the lattice polytope which is the convex hull of $\{\rho(e_1),\ldots,\rho(e_n)\}$ in $\RR^N$.  One has $\dim \Pc_G = N - 1$ if $G$ has at least one odd cycle, and $\dim \Pc_G = N - 2$ if $G$ is bipartite.
%
%
%
The \textit{edge ring} $K[G]$ of $G$ is the toric ring of $\Pc_G$, that is, $K[G]=K[\Pc_G]$ and the toric ideal $I_G$ of $K[G]$ is the toric ideal of $\Pc_{G}$, that is, $I_{G} = I_{\Pc_{G}}$. 

Recall from \cite{2linear} what a system of generators of the toric ideal $I_G$ is.  A \textit{walk} of $G$ {\em of length $q$} connecting $v_1 \in V(G)$ and $v_{q+1} \in V(G)$ is a finite sequence of the form
\[
\Gamma=(\{v_1,v_2\},\{v_2,v_3\},\ldots,\{v_q,v_{q+1}\})
\]
with each $\{v_k, v_{k+1}\} \in E(G)$.  An \textit{even walk} is a walk of even length and a \textit{closed walk} is a walk such that $v_1=v_{q+1}$.  Given an even closed walk
\[
\Gamma=(e_{i_1},e_{i_2},\ldots,e_{i_{2q}})
\]
of $G$ with each $e_k \in E(G)$,
we write $f_{\Gamma}$ for the binomial
$$f_{\Gamma}=\prod_{k=1}^{q}x_{i_{2k-1}}-\prod_{k=1}^{q}x_{i_{2k}}$$
belonging to $I_G$,
where $\pi(x_i)=\tb^{\rho(e_i)}s$.

\begin{Lemma}
\label{evenwalk}
The toric ideal $I_G$ of a finite connected simple graph $G$ is generated by those binomials $f_{\Gamma}$, where $\Gamma$ is an even closed walk of $G$.
\end{Lemma}

As a result, it follows that

\begin{Lemma}\label{lem:gen}
\textnormal{(a)} A finite connected simple graph $G$ has a $4$-cycle if and only if a homogeneous polynomial of $K[{\bf x}]$ of degree $2$ belongs to $I_G$.  

\textnormal{(b)} Let $\Gamma=(e_{i_1},\ldots,e_{i_6})$ be an even closed walk of $G$ of length $6$ with $f_{\Gamma} \in I_G$.
Then $\Gamma$ is either a $6$-cycle $C_6$ or the following{\rm :}
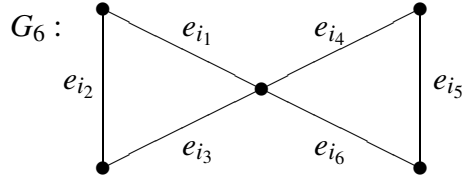
\begin{center}
	\begin{picture}(140,90)(90,90)
	\put(55,150){$G_6:$}
\put(75,130){$e_{i_2}$}
\put(170,150){$e_{i_4}$}
\put(170,105){$e_{i_6}$}
\put(215,130){$e_{i_5}$}
\put(120,105){$e_{i_3}$}
\put(120,150){$e_{i_1}$}
\put(90,160){\circle*{5}}
\put(90,100){\circle*{5}}
\put(150,130){\circle*{5}}
\put(210,160){\circle*{5}}
\put(210,100){\circle*{5}}
\put(90,160){\line(0,-1){60}}
\put(210,160){\line(0,-1){60}}
\put(90,160){\line(2,-1){60}}
\put(90,100){\line(2,1){60}}
\put(210,160){\line(-2,-1){60}}
\put(210,100){\line(-2,1){60}}
\end{picture}
\end{center}
\end{Lemma}

\section{Regularity and q-linear resolutions}


Let $S = K[x_{1}, \ldots, x_{n}]$ denote the polynomial ring in $n$ variables over a field $K$ with each $\deg x_{i} = 1$.  Let $I \subset S$ be a homogeneous ideal of $S$ and 
\[
{\bf F}_{S/I} : 0 \to \bigoplus_{j \geq 1} S(-a_{h_{j}})^{\beta_{h, j}} \to \cdots \to \bigoplus_{j \geq 1} S(-a_{1_{j}})^{\beta_{1, j}} \to S \to S/I \to 0
\]
a (unique) graded minimal free $S$-resolution of $S/I$.  The ({\em Castelnuovo-Mumford}\,) {\em regularity} of $S/I$ is 
\[
\reg (S/I) = \max\{ j - i : \beta_{i, j} \neq 0 \}. 
\] 
We say that $S/I$ has a {\em $q$-linear resolution} if $\beta_{i, j} = 0$ for each $1 \leq i \leq h$ and for each $j \neq q + i - 1$.  If $S/I$ has a $q$-linear resolution, then $\reg(S/I) = q - 1$ and $I$ is generated by homogeneous polynomials of degree $q$.  We refer the reader to, e.g., \cite{BH} and \cite{HH} for the detailed information about regularity and linear resolutions.  

%
%
%
%
%

\begin{Lemma}[{\cite[Proposition 1.7 (d)]{EG}}]
	\label{lem:EG}
If $S/I$ is Cohen--Macaulay and has a $q$-linear resolution, and if $c$ is the codimension of $S/I$, then the number of generators of $I$ is $\binom{c+q-1}{c-1}=\binom{c+q-1}{q}$.
\end{Lemma}

\section{Regularity of toric rings and degrees of lattice polytopes}
Recall that a matrix $A \in \ZZ^{N \times N}$ is $unimodular$ if $\det(A)=\pm 1$.  Given lattice polytopes $\mathcal{P} \subset \RR^N$ and $\mathcal{Q} \subset \RR^N$, we say that $\mathcal{P}$ and $\mathcal{Q}$ are {\em unimodularly equivalent} if there exist a unimodular matrix $U \in \ZZ^{N \times N}$ and a lattice point $\bold{w} \in \ZZ^N$ such that $\mathcal{Q}=f_U(\mathcal{P})+\bold{w}$, where $f_U$ is the linear transformation of $\RR^N$ defined by $U$, i.e., $f_U(\bold{v})=\bold{v}U$ for all $\bold{v} \in \RR^N$.  If $\Pc$ and $\Qc$ are unimodularly equivalent, then $K[\Pc] \cong K[\Qc]$.

Let $\Pc \subset \RR^N$ be a lattice polytope of dimension $d$.  The {\em $\delta$-polynomial} of $\Pc$ is the polynomial
\[
\delta(\Pc,\lambda)=(1-\lambda)^{d+1} \left[ 1+ \sum_{t=1}^{\infty} |t\Pc \cap \ZZ^N | \lambda^t \right]
\]
in $\lambda$, where $t\Pc=\{t \ab : \ab \in \Pc \}$.  Each coefficient of $\delta(\Pc, \lambda)$ is a nonnegative integer and the degree of $\delta(\Pc, \lambda)$ is at most $d$.  If $\Pc$ and $\Qc$ are unimodularly equivalent, then $\delta(\Pc,\lambda)=\delta(\Qc,\lambda)$.  Let $\deg(\Pc)$ denote the degree of $\delta(\Pc,\lambda)$ and set $\codeg(\Pc)=d+1-\deg(\Pc)$.
It then follows that
\[
\codeg(\Pc)= \min \{r \in \ZZ_{\geq 1} : \text{int}(r\Pc)\cap \ZZ^{N} \neq \emptyset \},
\]
where $\text{int}(\Pc)$ is the relative interior of $\Pc$ in $\RR^N$ and where $\ZZ_{\geq 1}$ stands for the set of positive integers.  We refer the reader to \cite[Part II]{HibiRedBook} for the detailed information about $\delta$-polynomials and their related topics.

\begin{Lemma}[{\cite[Stanley's Monotonicity-Theorem]{S}}]
	\label{lem:sta}
	Let $\Qc \subset \Pc \subset \RR^N$ be lattice polytopes. 
	Then $\deg(\Qc) \leq \deg(\Pc)$. 
\end{Lemma}

\begin{Corollary}
	\label{cor:deg}
	Let $G$ be a finite connected simple graph and let $G' \subset G$ be a connected subgraph of $G$.  Then $\deg(\Pc_{G'}) \leq \deg (\Pc_{G})$. 
\end{Corollary}

In general we say that a convex polytope of dimension $d$ is {\em full dimensional} if it embeds in $\RR^d$.
A full dimensional lattice polytope $\Pc \subset \RR^d$ satisfying the condition  
\begin{eqnarray}
\label{Boston}
\ZZ^{d+1} = \sum_{\ab \in \Pc \cap \ZZ^d} \ZZ (\ab, 1)
\end{eqnarray}
will be particularly of interest. 

\begin{Lemma}[{\cite[p. 6]{HKN}}]
	\label{lem:span}
	Let $\Pc \subset \RR^{d}$ be a full dimensional lattice polytope 
	which satisfies the condition $(\ref{Boston})$. 
	Then $\reg (K[\Pc]) \ge \deg (\Pc)$. 
\end{Lemma}

\begin{Corollary}
	\label{lem:reg}
	Let $G$ be a finite connected simple graph on $[N]$.
	Then 
	\[
	\reg (K[\Pc_G]) \ge \deg (\Pc_G).
	\] 
\end{Corollary}
\begin{proof}
First, suppose that $G$ has at least one odd cycle and write $\Hc \subset \RR^N$ for the affine subspace spanned by $\Pc_G$.  We define the affine transformation $\Psi : \Hc \to \RR^{N-1}$ by setting 
	\[
	\Psi(a_1, a_2, \ldots, a_{N-1}, a_{N}) = (a_1, a_2, \ldots, a_{N-1}).
	\]
Then $\Psi(\Hc \cap \ZZ^N) = \ZZ^{N-1}$.  The image $\Psi(\Pc_G)$ of $\Pc_G$ is a full dimensional lattice polytope of $\RR^{N-1}$ with $K[\Pc_G] \iso K[\Psi(\Pc_G)]$ and $\delta(\Pc_G,\lambda) = \delta(\Psi(\Pc_G),\lambda)$.  Let $H$ be a connected spanning subgraph of $G$ such that $H$ has exactly $N$ edges and that $H$ has exactly one cycle which is odd. Note that we can obtain such a connected spanning subgraph $H$ by taking a spanning tree of $G$ and adding an edge of $G$ to the spanning tree.   Then $\Psi(\Pc_H) \subset \RR^{N-1}$ is a full dimensional simplex which is unimodularly equivalent to the full dimensional standard simplex of $\RR^{N-1}$.  Hence $\Psi(\Pc_H)$ satisfies the condition $(\ref{Boston})$.  In particular $\Psi(\Pc_G)$ satisfies the condition $(\ref{Boston})$.  Lemma \ref{lem:span} now guarantees $\reg (K[\Pc_G]) \ge \deg (\Pc_G)$, as desired.

Second, suppose that $G$ is a bipartite graph with the partition $[N]=U \cup V$, where $1 \in U$ and $N \in V$.  Let $\Hc \subset \RR^N$ denote the affine subspace spanned by $\Pc_G$ and define the affine transformation $\Psi' : \RR^N \to \RR^{N-2}$ by setting 
	\[
	\Psi(a_1, a_2, \ldots, a_{N-1}, a_{N}) = (a_2, \ldots, a_{N-1}).
	\]
Then $\Psi'(\Hc \cap \ZZ^N) = \ZZ^{N-2}$.  The image $\Psi'(\Pc_G)$ of $\Pc_G$ is a full dimensional lattice polytope of $\RR^{N-2}$ with $K[\Pc_G] \iso K[\Psi'(\Pc_G)]$ and $\delta(\Pc_G,\lambda) = \delta(\Psi'(\Pc_G),\lambda)$.  Let $T$ be a spanning tree of $G$.  Then $\Psi'(\Pc_T) \subset \RR^{N-2}$ is a full dimensional simplex which is unimodularly equivalent to the full dimensional standard simplex of $\RR^{N-2}$.  Hence $\Psi'(\Pc_T)$ satisfies the condition $(\ref{Boston})$.  In particular $\Psi'(\Pc_G)$ satisfies the condition $(\ref{Boston})$.  Lemma \ref{lem:span} now guarantees $\reg (K[\Pc_G]) \ge \deg (\Pc_G)$, as required.
\end{proof}

\section{Proof of main theorem}
In order to prove Theorem \ref{thm:lin}, a few lemmata must be prepared.
Let, as before, $G$ be a finite connected simple graph on $[N]$. 
Recall that if $K[G]$ has a $3$-linear resolution, then one has $\reg(K[G]) = 2 \geq  \deg(\Pc_G)$ from Corollary \ref{lem:reg} and $I_{G}$ is generated by homogeneous polynomials of degree $3$.  Especially, $G$ has no $4$-cycle and has no even cycle of length $\geq 8$ without chord (\cite[Lemma 3.3]{2linear}).
By computing $\deg(\Pc_G)$, we consider how many even closed walks as in Lemma \ref{lem:gen} (b) there can exist in $G$ when  $K[G]$ has a $3$-linear resolution. In particular, we give a upper bound of the number of generators of $I_G$.
First, we focus on $G_6$.
\begin{Lemma}
	\label{lem:2odd}
If $G$ has at least two disjoint odd cycles, then $\deg (\Pc_{G}) \geq 3$.
\end{Lemma}

\begin{proof}
Let $C$ be a cycle of $G$ of length $2k+1$ and $C'$ a cycle of $G$ of length $2\ell+1$ for which $V(C) \cap V(C') = \emptyset$.  We introduce the lattice polytope $\Qc \subset \RR^{2k+2\ell+2}$ which is the convex hull of $\{ \rho(e) \in \RR^{2k+2\ell +2} : e \in E(C) \cup E(C') \}$.  Since $\dim \Qc = 2k + 2\ell + 1$ and since $(1,1, \ldots, 1)$ belongs to $\mathrm{int}((k+\ell+1)\Qc) \cap \ZZ^{2k+2\ell+2}$, it follows that $\codeg (\Qc) \leq k + \ell + 1$ and $\deg (\Qc) \geq k + \ell + 1 \geq 3$.  By using Lemma \ref{lem:sta}, one has $\deg (\Pc_{G}) \geq 3$, as desired.
\end{proof}
From this lemma if $K[G]$ has a $3$-linear resolution, then any two odd cycles of $G$ have a common vertex.
In particular, we get the following corollary.
\begin{Corollary}
\label{FHM} 
If $\deg (\Pc_{G}) \leq 2$, then any two odd cycles of $G$ have a common vertex.  In particular, $G$ satisfies the odd cycle condition $($\cite[p.~410]{OH}$)$ and $K[G]$ is Cohen--Macaulay. 
\end{Corollary}
Hence, if $K[G]$ has a $3$-linear resolution, then $K[G]$ is Cohen-Macaulay.

Now, we show how many subgraphs of the form $G_6$ there exist in $G$, when  $K[G]$ has a $3$-linear resolution. In fact, the following lemma implies that $G$ has at most one subgraph of the form $G_6$.
\begin{Lemma}
\label{star}
If $G$ has at least three $3$-cycles and has no $4$-cycle, then $\deg(\Pc_{G}) \geq 3$.
\end{Lemma}

\begin{proof}
By virtue of Lemma \ref{lem:2odd}, we may assume that $G$ has a subgraph $G'$ which is the union of three $3$-cycles $C_1, C_2$ and $C_3$ for which $V(C_i) \cap V(C_j) \neq \emptyset$ for $1 \leq i < j \leq 3$.  Since $G$ has no $4$-cycle, one has $E(C_i) \cap E(C_j) = \emptyset$ for $1 \leq i < j \leq 3$.  Let $i_1, i_2$ and $i_3$ be vertices of $G$ with $i_1,i_2 \in V(C_1)$, $i_2,i_3 \in V(C_2)$ and $i_1, i_3 \in V(C_3)$.  It then follows that either $i_1=i_2=i_3$ or $|\{i_1, i_2, i_3\}| = 3$.  However, if $|\{i_1, i_2, i_3\}| = 3$, then $G$ has a $4$-cycle.  As a result, one has $i_1=i_2=i_3$ and $G'$ is the following graph:  
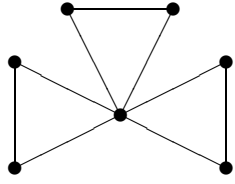
\begin{center}
	\begin{picture}(140,90)(90,100)
	\put(130,170){\circle*{5}}
	\put(170,170){\circle*{5}}
	\put(150,130){\circle*{5}}
	\put(110,150){\circle*{5}}
	\put(110,110){\circle*{5}}
	\put(190,150){\circle*{5}}
	\put(190,110){\circle*{5}}
	\put(150,130){\line(-1,2){20}}
	\put(150,130){\line(1,2){20}}
	\put(150,130){\line(2,-1){40}}
	\put(150,130){\line(2,1){40}}
	\put(150,130){\line(-2,-1){40}}
	\put(150,130){\line(-2,1){40}}
	\put(110,150){\line(0,-1){40}}
		\put(190,150){\line(0,-1){40}}
			\put(130,170){\line(1,0){40}}
	\end{picture}
\end{center}
	Since $\deg(\Pc_{G'})=3$, by using Corollary \ref{cor:deg}, one has $\deg (\Pc_{G}) \geq 3$, as required.
\end{proof}

Next, we focus on $C_6$.
The following lemma implies that $G$ has no even cycles expect for $6$-cycles as subgraph when $K[G]$ has a $3$-linear resolution.

\begin{Lemma}
	\label{lem:even}
	If $G$ has an even cycle of length $\geq 8$ with a chord and has no $4$-cycle, then $\deg(\Pc_G) \geq 3$.
\end{Lemma}

\begin{proof}
Let $C_{k,\ell}$ denote the subgraph of $G$ with
\[
E(C_{k,\ell}) = \{\{1,2\}, \{2,3\},\ldots,\{2k-1,2k\},\{1,2k\}, \{1,\ell\} \},
\]
where $k \geq 4$ and $3 \leq \ell \leq k+1$.  One has   
	\begin{displaymath}
		\dim (\Pc_{C_{k,\ell}})=\begin{cases}
		2k-1 &(\text{if\ } \ell \text{\ is odd}),\\
		2k-2 &(\text{if\ } \ell \text{\ is even}).
		\end{cases}
	\end{displaymath}
Since $(\eb_1+\cdots+\eb_{2k})+(\eb_1+\eb_{\ell}) \in \text{int}((k+1)\Pc_{C_{k,\ell}}) \cap \ZZ^{2k}$, it follows that
	\begin{displaymath}
	\deg (\Pc_{C_{k,\ell}})=\begin{cases}
	k-1 &(\text{if\ } \ell \text{\ is odd}),\\
	k-2 &(\text{if\ } \ell \text{\ is even}).
	\end{cases}
	\end{displaymath}
	Since $G$ has no $4$-cycle, one has $(k,\ell) \neq (4,4)$.  Hence $\deg(\Pc_{C_{k,\ell}}) \geq 3$.  Again Corollary \ref{cor:deg} guarantees that $\deg (\Pc_{G}) \geq 3$, as desired.
\end{proof}

Now, we show how many subgraphs of the form $C_6$ there exist in $G$, when  $K[G]$ has a $3$-linear resolution. In fact, the following lemma and Lemma \ref{lem:even} imply that $G$ has at most one subgraph of the form $C_6$.

\begin{Lemma}
	\label{lem:6cycle}
	Suppose that
	\begin{enumerate}
		\item[{\rm (i)}] $G$ has at least two $6$-cycles{\rm ;}
		\item[{\rm (ii)}] $G$ has no $4$-cycle{\rm ;}
		\item[{\rm (iii)}] $G$ has no even cycle of length $\geq 8$. 
	\end{enumerate}
Then $\deg(\Pc_G) \geq 3$.
\end{Lemma}

\begin{proof}

Take different two $6$-cycles $C_1$ and $C_2$  with the vertex sets $V(C_1)=\{1,\ldots,6\} \subset [N]$ and $V(C_2)=\{v_1,\ldots,v_6\} \subset [N]$ and the edge sets $E(C_1)=\{\{1,2\},\ldots,\{5,6\},\{1,6\}\}$ and   $E(C_2)=\{\{v_1,v_2\},\ldots,\{v_5,v_6\},\{v_1,v_6\}\}.$
Let $H$ be the finite simple graph such that
$V(H)=V(C_1) \cup V(C_2)$ and $E(H)=E(C_1) \cup E(C_2)$.
Then we can assume that $H$ is a subgraph of $G$.

Suppose that $V(C_1)\cap V(C_2) = \emptyset$.
We introduce the lattice polytope $\Qc \subset \RR^N$  which is the convex hull of $\{\rho(e) \in \RR^N : e \in E(C_1) \cup E(C_2) \}$.
Then one has $\deg(\Qc) = 4$.
Hence by using Lemma \ref{lem:sta}, we obtain $\deg(\Pc_G) \geq 3$.

Now, we show that if $V(C_1)\cap V(C_2) \neq \emptyset$,
then $\deg(\Pc_H) \geq 3$, hence one has $\deg(\Pc_G) \geq 3$ from Corollary \ref{cor:deg}.

First, suppose that $|V(C_1) \cap V(C_2)|=1$. 
Then we can assume that $v_1=1$.
Moreover, one has  $\deg(\Pc_H) =4$.

Second, suppose that $|V(C_1) \cap V(C_2)|=2$.
Then we should consider the cases where
$(v_1,v_2)=(1,2)$,
$(v_1,v_2)=(1,3)$,
$(v_1,v_2)=(1,4)$,
$(v_1,v_3)=(1,3)$,
$(v_1,v_3)=(1,4)$
or
$(v_1,v_4)=(1,4)$.
In these cases, it follows that  $H$ has $4$-cycles, $H$ has even cycles of length $\geq 8$ or  $\deg(\Pc_H) \geq 3$.

Finally, suppose that $|V(C_1) \cap V(C_2)|=k$ with some positive integer $3 \leq k \leq 6$ and $v_1=1$.
We consider $H$ for any $\{i_2,\ldots,i_{k}\} \subset \{2,\ldots,6\}$ with $i_2<\cdots<i_{k}$ and $\{j_2,\ldots,j_{k}\} \subset \{2,\ldots,6\}$ such that $(v_{i_2},\ldots,v_{i_{k}})=(j_2,\ldots,j_{k})$.
In many cases, $H$ has $4$-cycles or $H$ has even cycles of length $\geq 8$.
In fact, $H$ is one of the following graphs up to graph equivalence:
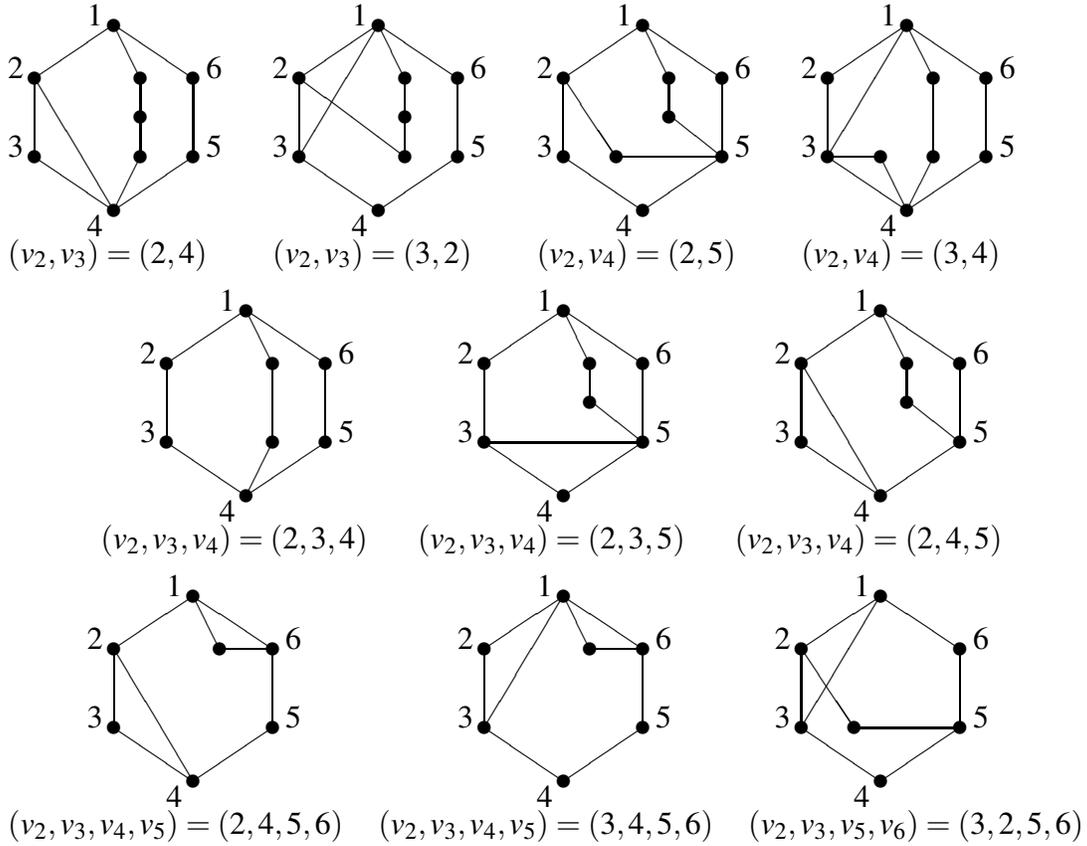
\begin{center}
	\begin{picture}(400,100)
	\put(10,70){\circle*{5}}
	\put(10,40){\circle*{5}}
	\put(40,90){\circle*{5}}
	\put(40,20){\circle*{5}}
	\put(70,70){\circle*{5}}
	\put(70,40){\circle*{5}}
	
		\put(0,70){$2$}
	\put(0,40){$3$}
	\put(30,90){$1$}
	\put(30,10){$4$}
	\put(75,70){$6$}
	\put(75,40){$5$}

	\put(10,70){\line(0,-1){30}}
	\put(70,70){\line(0,-1){30}}
	\put(10,70){\line(3,2){30}}
	\put(70,70){\line(-3,2){30}}
	\put(40,20){\line(3,2){30}}
	\put(40,20){\line(-3,2){30}}

		\put(50,55){\circle*{5}}
	\put(50,40){\circle*{5}}
	\put(50,70){\circle*{5}}
	\put(10,70){\line(3,-5){30}}
	\put(40,20){\line(1,2){10}}
	\put(50,40){\line(0,1){30}}
	\put(40,90){\line(1,-2){10}}

	\put(0,0){$(v_2,v_3)=(2,4)$}
	
		\put(100,70){$2$}
	\put(100,40){$3$}
	\put(130,90){$1$}
	\put(130,10){$4$}
	\put(175,70){$6$}
	\put(175,40){$5$}
	
		\put(110,70){\circle*{5}}
	\put(110,40){\circle*{5}}
	\put(140,90){\circle*{5}}
	\put(140,20){\circle*{5}}
	\put(170,70){\circle*{5}}
	\put(170,40){\circle*{5}}
	\put(150,55){\circle*{5}}
	\put(150,40){\circle*{5}}
	\put(150,70){\circle*{5}}
	\put(110,70){\line(0,-1){30}}
	\put(170,70){\line(0,-1){30}}
	\put(110,70){\line(3,2){30}}
	\put(170,70){\line(-3,2){30}}
	\put(140,20){\line(3,2){30}}
	\put(140,20){\line(-3,2){30}}
	\put(140,90){\line(-3,-5){30}}
	\put(150,40){\line(0,1){30}}
	\put(140,90){\line(1,-2){10}}
	\put(110,70){\line(4,-3){40}}
	\put(100,0){$(v_2,v_3)=(3,2)$}
	
		\put(210,70){\circle*{5}}
	\put(210,40){\circle*{5}}
	\put(240,90){\circle*{5}}
	\put(240,20){\circle*{5}}
	\put(270,70){\circle*{5}}
	\put(270,40){\circle*{5}}
	\put(230,40){\circle*{5}}
	\put(250,55){\circle*{5}}
	\put(250,70){\circle*{5}}
	
		\put(200,70){$2$}
	\put(200,40){$3$}
	\put(230,90){$1$}
	\put(230,10){$4$}
	\put(275,70){$6$}
	\put(275,40){$5$}
	
	\put(210,70){\line(0,-1){30}}
	\put(270,70){\line(0,-1){30}}
	\put(210,70){\line(3,2){30}}
	\put(270,70){\line(-3,2){30}}
	\put(240,20){\line(3,2){30}}
	\put(240,20){\line(-3,2){30}}
	
	\put(240,90){\line(1,-2){10}}
	\put(230,40){\line(1,0){40}}
	\put(250,55){\line(0,1){15}}
	\put(210,70){\line(2,-3){20}}
		\put(270,40){\line(-5,4){20}}
	\put(200,0){$(v_2,v_4)=(2,5)$}
	
		\put(310,70){\circle*{5}}
	\put(310,40){\circle*{5}}
	\put(340,90){\circle*{5}}
	\put(340,20){\circle*{5}}
	\put(370,70){\circle*{5}}
	\put(370,40){\circle*{5}}
	\put(330,40){\circle*{5}}
	\put(350,40){\circle*{5}}
	\put(350,70){\circle*{5}}
		\put(300,70){$2$}
	\put(300,40){$3$}
	\put(330,90){$1$}
	\put(330,10){$4$}
	\put(375,70){$6$}
	\put(375,40){$5$}
	
	\put(310,70){\line(0,-1){30}}
	\put(370,70){\line(0,-1){30}}
	\put(310,70){\line(3,2){30}}
	\put(370,70){\line(-3,2){30}}
	\put(340,20){\line(3,2){30}}
	\put(340,20){\line(-3,2){30}}
	\put(340,90){\line(-3,-5){30}}
	\put(310,40){\line(1,0){20}}
	\put(340,20){\line(1,2){10}}
	\put(330,40){\line(1,-2){10}}
	\put(350,40){\line(0,1){30}}
	\put(340,90){\line(1,-2){10}}
	\put(300,0){$(v_2,v_4)=(3,4)$}
	
	\end{picture}
\end{center}

\begin{center}
	\begin{picture}(400,100)
	\put(60,70){\circle*{5}}
	\put(60,40){\circle*{5}}
	\put(90,90){\circle*{5}}
	\put(90,20){\circle*{5}}
	\put(120,70){\circle*{5}}
	\put(120,40){\circle*{5}}
	\put(100,40){\circle*{5}}
	\put(100,70){\circle*{5}}
	\put(60,70){\line(0,-1){30}}
	\put(120,70){\line(0,-1){30}}
	\put(60,70){\line(3,2){30}}
	\put(120,70){\line(-3,2){30}}
	\put(90,20){\line(3,2){30}}
	\put(90,20){\line(-3,2){30}}
	\put(100,40){\line(0,1){30}}
	\put(90,90){\line(1,-2){10}}
	\put(90,20){\line(1,2){10}}
	\put(35,0){$(v_2,v_3,v_4)=(2,3,4)$}
	
		\put(50,70){$2$}
	\put(50,40){$3$}
	\put(80,90){$1$}
	\put(80,10){$4$}
	\put(125,70){$6$}
	\put(125,40){$5$}
	
		\put(170,70){$2$}
	\put(170,40){$3$}
	\put(200,90){$1$}
	\put(200,10){$4$}
	\put(245,70){$6$}
	\put(245,40){$5$}
	
	\put(180,70){\circle*{5}}
	\put(180,40){\circle*{5}}
	\put(210,90){\circle*{5}}
	\put(210,20){\circle*{5}}
	\put(240,70){\circle*{5}}
	\put(240,40){\circle*{5}}
	\put(220,55){\circle*{5}}
	\put(220,70){\circle*{5}}
	\put(180,70){\line(0,-1){30}}
	\put(240,70){\line(0,-1){30}}
	\put(180,70){\line(3,2){30}}
	\put(240,70){\line(-3,2){30}}
	\put(210,20){\line(3,2){30}}
	\put(210,20){\line(-3,2){30}}
	
	\put(210,90){\line(1,-2){10}}
	\put(180,40){\line(1,0){60}}
	\put(220,55){\line(0,1){15}}
	\put(240,40){\line(-5,4){20}}
		\put(155,0){$(v_2,v_3,v_4)=(2,3,5)$}
	
	\put(300,70){\circle*{5}}
	\put(300,40){\circle*{5}}
	\put(330,90){\circle*{5}}
	\put(330,20){\circle*{5}}
	\put(360,70){\circle*{5}}
	\put(360,40){\circle*{5}}
	
		\put(290,70){$2$}
	\put(290,40){$3$}
	\put(320,90){$1$}
	\put(320,10){$4$}
	\put(365,70){$6$}
	\put(365,40){$5$}
	
	\put(340,55){\circle*{5}}
	\put(340,70){\circle*{5}}
	\put(300,70){\line(0,-1){30}}
	\put(360,70){\line(0,-1){30}}
	\put(300,70){\line(3,2){30}}
	\put(360,70){\line(-3,2){30}}
	\put(330,20){\line(3,2){30}}
	\put(330,20){\line(-3,2){30}}
	\put(300,70){\line(3,-5){30}}
\put(330,90){\line(1,-2){10}}
\put(340,55){\line(0,1){15}}
\put(360,40){\line(-5,4){20}}
		\put(275,0){$(v_2,v_3,v_4)=(2,4,5)$}
	
	\end{picture}
\end{center}
\begin{center}
	\begin{picture}(400,100)
	\put(40,70){\circle*{5}}
	\put(40,40){\circle*{5}}
	\put(70,90){\circle*{5}}
	\put(70,20){\circle*{5}}
	\put(100,70){\circle*{5}}
	\put(100,40){\circle*{5}}
	\put(80,70){\circle*{5}}
	\put(40,70){\line(0,-1){30}}
	\put(100,70){\line(0,-1){30}}
	\put(40,70){\line(3,2){30}}
	\put(100,70){\line(-3,2){30}}
	\put(70,20){\line(3,2){30}}
	\put(70,20){\line(-3,2){30}}
	\put(70,90){\line(1,-2){10}}
	\put(80,70){\line(1,0){20}}
		\put(40,70){\line(3,-5){30}}
		
			\put(30,70){$2$}
		\put(30,40){$3$}
		\put(60,90){$1$}
		\put(60,10){$4$}
		\put(105,70){$6$}
		\put(105,40){$5$}
	
	\put(0,0){$(v_2,v_3,v_4,v_5)=(2,4,5,6)$}
	
	\put(180,70){\circle*{5}}
	\put(180,40){\circle*{5}}
	\put(210,90){\circle*{5}}
	\put(210,20){\circle*{5}}
	\put(240,70){\circle*{5}}
	\put(240,40){\circle*{5}}
	\put(220,70){\circle*{5}}
	\put(180,70){\line(0,-1){30}}
	\put(240,70){\line(0,-1){30}}
	\put(180,70){\line(3,2){30}}
	\put(240,70){\line(-3,2){30}}
	\put(210,20){\line(3,2){30}}
	\put(210,20){\line(-3,2){30}}
	
	\put(210,90){\line(-3,-5){30}}
		\put(210,90){\line(1,-2){10}}
		\put(220,70){\line(1,0){20}}
		
				\put(170,70){$2$}
		\put(170,40){$3$}
		\put(200,90){$1$}
		\put(200,10){$4$}
		\put(245,70){$6$}
		\put(245,40){$5$}
	
	\put(140,0){$(v_2,v_3,v_4,v_5)=(3,4,5,6)$}
	
	\put(300,70){\circle*{5}}
	\put(300,40){\circle*{5}}
	\put(330,90){\circle*{5}}
	\put(330,20){\circle*{5}}
	\put(360,70){\circle*{5}}
	\put(360,40){\circle*{5}}
	
	\put(320,40){\circle*{5}}
	\put(300,70){\line(0,-1){30}}
	\put(360,70){\line(0,-1){30}}
	\put(300,70){\line(3,2){30}}
	\put(360,70){\line(-3,2){30}}
	\put(330,20){\line(3,2){30}}
	\put(330,20){\line(-3,2){30}}
	\put(300,70){\line(2,-3){20}}
		\put(320,40){\line(1,0){40}}
		
				\put(290,70){$2$}
		\put(290,40){$3$}
		\put(320,90){$1$}
		\put(320,10){$4$}
		\put(365,70){$6$}
		\put(365,40){$5$}

	\put(330,90){\line(-3,-5){30}}
	\put(280,0){$(v_2,v_3,v_5,v_6)=(3,2,5,6)$}
	\end{picture}
\end{center}
In the above cases, one has $\deg(\Pc_H) \geq 3$.

Therefore, we complete the proof.
\end{proof}


Finally, we are in the position to give a proof of Theorem \ref{thm:lin}. 

\begin{proof}[Proof of Theorem \ref{thm:lin}]
Suppose that $K[G]$ has a $3$-linear resolution.  Then the above discussion guarantees that $K[G]$ is Cohen--Macaulay.  
Lemma \ref{lem:EG} then says that the number of generators of $I_G$ is ${c + 2 \choose 3} = c(c+1)(c+2)/6$, where $c = |E(G)| - \dim K[G]$.   
 On the other hand, from Lemma \ref{lem:gen} (b) the number of generators of $I_G$ is at most $2$.
   If $K[G]$ is not a hypersurface, then the number of generators of $I_G$ is equal to $2$.  However, no positive integer $c$ satisfies $c(c+1)(c+2)/6 = 2$.  As a result, $K[G]$ must be a hypersurface.
\end{proof}

\begin{Example}
{\em
Let $G_{2, 3}$ be the complete bipartite graph on $\{1,2\}\cup\{3,4,5\}$.  Its edge ring 
\[
K[G_{2,3}] \cong K[x_{1}, x_{2}, x_{3}, x_{4}, x_{5}, x_{6}] / (x_{1}x_{5} - x_{2}x_{4}, x_{1}x_{6} - x_{3}x_{4}, x_{2}x_{6} - x_{3}x_{5})
\]
has a 2-linear resolution (\cite[Theorem 4.6]{2linear}).  However, $K[G_{2,3}]$ is not a hypersurface.
} 
\end{Example}


\begin{thebibliography}{99}

\bibitem{BH}
W.~Bruns and J.~Herzog, 
Cohen-Macaulay Rings, Revised Ed., 
Cambridge Stud. Adv. Math., vol. {\bf 39}, 
Cambridge University Press, Cambridge, 1998. 

\bibitem{EG}
D.~Eisenbud and S.~Goto, 
Linear free resolutions and minimal multiplicity, 
{\em J. Algebra} {\bf 88} (1984), 89--133. 

\bibitem{HH}
J.~Herzog and T.~Hibi, 
Monomial ideals, 
Graduate Texts in Mathematics {\bf 260}, 
Springer, London, 2010. 

\bibitem{HibiRedBook}
T.~Hibi,
``Algebraic Combinatorics on Convex Polytopes,"
Carslaw Publications, Glebe, N.S.W., Australia, 1992.


\bibitem{HKN}
J.~Hofscheier, L.~Katth\"{a}n and B.~Nill, 
Ehrhart theory of spanning lattice polytopes, 
{\em Int. Math. Res. Not. IMRN}, to appear. 


\bibitem{OH}
H.~Ohsugi and T.~Hibi, 
Normal polytopes arising from finite graphs, 
{\em J. Algebra} {\bf 207} (1998), 409--426. 

\bibitem{2linear}
H.~Ohsugi and T.~Hibi, 
Toric ideals generated by quadratic binomials, 
{\em J. Algebra} {\bf 218} (1999), 509--527. 

\bibitem{S}
R.~P.~Stanley, 
A monotonicity property of $h$-vectors and $h^{*}$-vectors, 
{\em European J. Combin.} {\bf 14} (1993), 251--258. 

\bibitem{Stu}
B.~Sturmfels, 
``Gr\"{o}bner bases and convex polytopes", 
Amer. Math. Soc., Providence, RI, 1996.


\end{thebibliography}
\end{document}